\documentclass[12pt]{article}
\usepackage{amssymb,amsmath, amsthm}
\usepackage{color,xcolor}
\usepackage[normalem]{ulem}
\newcommand{\eb}{\begin{equation}}
\newcommand{\ee}{\end{equation}}
\newcommand{\ebx}{\begin{equation*}}
\newcommand{\eex}{\end{equation*}}
\newcommand{\delete}[1]{{\color{red}\ifmmode\text{\sout{\ensuremath{#1}}}\else\sout{#1}\fi}}

\newtheorem{lemma}{Lemma}[section]
\newtheorem{proposition}[lemma]{Proposition}
\newtheorem{theorem}[lemma]{Theorem}
\newtheorem{example}[lemma]{Example}
\newtheorem{corollary}[lemma]{Corollary}
\newtheorem{definition}[lemma]{Definition}

\allowdisplaybreaks

\makeatletter
\renewcommand*\env@matrix[1][*\c@MaxMatrixCols c]{%
  \hskip -\arraycolsep
  \let\@ifnextchar\new@ifnextchar
  \array{#1}}
\makeatother

\begin{document}
	
	\title{Hermitian Rank and Rigidity of Holomorphic Mappings}
	\author{Yun Gao\footnote{School of Mathematical Sciences, Shanghai Jiao Tong University, Shanghai, People's Republic of China. \textbf{Email:}~gaoyunmath@sjtu.edu.cn}, }

	\maketitle
	
	\begin{abstract}
	Huang's Lemma is an important tool in CR geometry to study rigidity
	problems. This paper introduces a generalization of Huang's Lemma 
		 based on the rigidity properties of holomorphic mappings preserving certain orthogonality on projective spaces, which is optimal for the case of partial linearity. By exploring the intricate relationship between rigidity  and Huang's Lemma, we establish that the rigidity properties of Segre maps or proper holomorphic mappings between generalized balls with Levi-degenerate boundaries can be inferred from those between generalized balls with Levi-non-degenerate boundaries by in a coordinate-free and more geometric manner.

		\end{abstract}

\section{Introduction}

The rigidity problem of mappings between hypersurfaces embedded in
complex Euclidean spaces has attracted much attention  since the
pioneer work of Poincar\'e (\cite {Po}) and Alexander (\cite{Al}). Early  work had been
mainly focused on holomorphic mappings between spheres of the same
dimension. In 1978, Webster in \cite{We}  studied   proper
holomorphic mappings between balls with complex co-dimension one.
Since then, important progress has been made towards the
understanding of proper holomorphic mappings between unit balls in
complex spaces of different dimensions. (See, for instance,
\cite{Fr,CS,St,Hu,HJY,HLTX,DX,NZ,GN, XY} and many references therein.)

In a striking paper \cite{Hu}, Huang proved that a holomorphic map
$F$ from $\mathbb B^n$ to $\mathbb B^N$ with $n\le N\le 2n-2$, that
is $C^2$-smooth up to the boundary, is a totally geodesic embedding
with respect to Bergman metrics, that solves a long standing open
question  proposed by Cima-Suffridge \cite{CS} and Forstneri\v{c}
\cite{Fr}. A fundamental novelty of the work in \cite{Hu} is the
discovery of a lemma nowadays called in the literature
\textit{Huang's Lemma} stated as follows:

\noindent\textbf{Huang's Lemma}~\cite{Hu}. {\it Let $\{\phi_j\}$ and
$\{\psi_j\}$ be holomorphic functions in $z\in \mathbb C^n$ ($n\ge
2$) near the origin. Assume that $\phi_j(0)=\psi_j(0)=0$. Let
$H(z,\bar z)$ be a complex-valued real-analytic function for $z
\approx 0$ such that
\begin{equation}\label{Huang}
\sum_{j=1}^k\phi_j(z)\overline{\psi_j(z)}=\|z\|^2H(z,\bar z).
\end{equation}
When $1\le k\le n-1$, then $H(z,\bar z)\equiv 0$. }

Huang's Lemma and its various versions have become an indispensable
tool in the later  studies  of holomorphic mappings and have been
used to obtain   many deep results. (See in \cite{BEH}, \cite{BH},
\cite{EHZ}, \cite{HLTX}, \cite{HJY2}, etc.) For instance, in a paper of
Ebenfelt-Shroff ~\cite{ES}, a generalized Huang's Lemma  was used to
derive various degeneracy of holomorphic mappings from a
Levi-nondegenerate CR hypersurface into a hyperquadric with small
signature difference. In a recent paper \cite{Xi}, Xiao considered
Huang's lemma  in the case of  $k=n$ and proved that if $H(z,\bar
z)\not\equiv 0$, then both vector-valued holomorphic functions
$\phi$ and $\psi$ are linear up to  holomorphic factors.
 This result  generalized for the first time  Huang's Lemma to the case
  where the relevant
  maps are not degenerate.

Recall that a real-valued real-analytic function $H(z,\bar z)$ of
finite rank  is the one which  can be written as $H(z,\bar
z)=\sum_{i=1}^p|f_i(z)|^2-\sum_{j=1}^q|g_j(z)|^2$, $p,q\in \mathbb
N$, where $f_i$ and $g_i$ are $\mathbb C$-linearly independent
holomorphic functions. Here $p+q$ is called the rank of $H(z,\bar
z)$. One can verify that the rank of $H(z,\bar z)$ is independent of
the choice of $f_i$'s and $g_i$'s. The concept of Hermitian ranks
has profound connections with various mapping problems in several
complex variables, which originated from  the Hilbert 17th problem.
For instance, Ebenfelt \cite{Eb} proposed the SOS (sums of squares)
conjecture, which relates Hermitian ranks with a well-known gap
conjecture of Huang-Ji-Yin \cite{HJY} for proper holomorphic
mappings between balls. The rank of $H(z,\bar z)$ is zero if and only if $H(z,\bar z)$ is identically zero, which implies the corresponding map is degenerate.

In this paper, we  study a new version  of Huang's Lemma in the case
of arbitrary signature and higher dimension difference by the method through orthogonal pairs introduced in \cite{Ga}. The starting
point of our approach to Huang's lemma is to homogenize and polarize
Eq(\ref{Huang}), which gives us an \textit{orthogonal pair}
(Definition \ref{orth}) $F:=(\phi,\psi)$ between two projection
spaces. An orthogonal pair is a pair of holomorphic maps preserving
certain orthogonality defined on  projective spaces using Hermitian
forms ( which may be degenerate), which is the generalization of Segre map. 
Hence, we can use the results about orthogonal pairs in \cite{Ga}. However, in \cite{Ga},
most cases handled were with a non-degenerate Hermitian form in the
source projective space. In this article, we need to handle
degenerate cases as degeneracy  always shows up when  new variables
are  introduced in the process of homogenization of Eq
(\ref{Huang}). Basing on the rigidities of orthogonal pairs, we can extend Huang's Lemma.

Our first main result is as follows:
\begin{theorem}\label{pair}
For  $r, s\in{\mathbb Z_{\ge 0}}, n'\in \mathbb N$, let
$\{\phi_j\}_{j=1}^{n'},$ $\{\psi_j\}_{j=1}^{n'}$ be holomorphic
functions in $ z\in \mathbb C^{r+s}$ near the origin and let
$H(z,\bar z)$ be a complex-valued real-analytic function such that
$$\sum_{j=1}^{n'}\phi_j(z)\overline{\psi_j(z)}=\|z\|_{r,s}^2H(z,\bar z).$$
 If $r+s\leq n'\le 2(r+s)-3$ and $H(z,\bar z) \not\equiv 0$, then there exist  holomorphic
functions $h_1(z)$ and $h_2(z)$ such that $H(z,\bar
z)=h_1(z)\overline{h_2(z)}$. Moreover there exist holomorphic
functions $\{\Phi_i\}^q_{i=1}$, $\{\Psi_i\}^q_{i=1}$, where
$q=n'-r-s$ and  matrices $B$, $C$ satisfying $\overline
C^TB=H_{r,s,q}$ such that
$$
[\phi_1,\cdots,\phi_{n'}]^T=B[h_1(z)z_1,\cdots, h_1(z)z_{r+s},\Phi_1,\cdots,\Phi_q]^T,
$$
$$
[\psi_1,\cdots,\psi_{n'}]^T=C[h_2(z)z_1,\cdots,  h_2(z)z_{r+s},\Psi_1,\cdots,\Psi_q]^T.
$$

\end{theorem}

Theorem \ref{pair} serves as a generalization of many well-known results. Specifically, when $n'=r+s$, it pertains to the result proved by Xiao in~\cite{Xi}. The proof of this theorem will be presented in Section 4. 

The following example which is called generalized Whitney map in \cite{HX} demonstrates that if $n'= 2(r+s) - 2$, then the result fails. Consequently, Theorem \ref{pair} is optimal for characterizing partial linearity.

\begin{example}For $r+s=n$ and $n'=2n-2$, let 
$$\left\{ \begin{array}{lll}
	\phi_j=\psi_j=z_1z_j,& & 1\le j\le r\\
	\phi_j=-\psi_j=z_1z_j,& & r+1\le j\le n-1,\\ 
	\phi_j=\psi_j=z_{n}z_{j+2-n},&& n \le j\le n+r-2\\
	\phi_j=-\psi_j=z_{n}z_{j+2-n},&& n+r-1\le j\le n'
\end{array}\right.$$
then
$$\sum_{j=1}^{n'}\phi_j(z)\overline{\psi_j(z)}=\|z\|_{r,s}^2(|z_1|^2+|z_{n}|)^2.$$
\end{example}

 Using the relationship between Huang's type lemma and rigidities of orthogonal pairs (respectively, orthogonal maps ), we prove the following result which explores the rigidities 
of holomorphic proper pairs (respectively,orthogonal map ) between generalized balls with Levi-degenerated boundaries can be infered from those between generalized balls with non-degenerated boundaries.

 To state the following theorem, we first introduce some notations.
In what follows, for $r,s,t\in\mathbb N$ and
$z=(z_1,\ldots,z_{r+s+t})\in\mathbb C^{r+s+t}$, we write $H_{r,s,t}$
for the diagonal matrix representing the (possibly degenerate and
indefinite) standard Hermitian form on $\mathbb C^{r+s+t}$ and write
the associated inner product and norm squared as follows
$$
\langle z,w\rangle_{r,s,t}:=z_1\bar w_1+\cdots+z_r\bar w_r
-z_{r+1}\bar w_{r+1}-\cdots-z_{r+s}\bar w_{r+s}
$$
$$
\|z\|_{r,s,t}^2:=|z_1|^2+\cdots+|z_r|^2-|z_{r+1}|^2-\cdots-|z_{r+s}|^2.
$$
We denote by $\mathbb C^{r,s,t}$ the $\mathbb C^{r+s+t}$ with the Hermitian form
defined above and by $\mathbb P^{r,s,t}:=\mathbb P\mathbb C^{r,s,t}$ its projectivization.
If $t=0$, we will simply write $\langle \cdot,\cdot\rangle_{r,s}$ and $\|\cdot\|_{r,s}$.
We write $\mathbb C^{r,s}$ and  $\mathbb P^{r,s}$  instead of $\mathbb C^{r,s,0}$ and $\mathbb P^{r,s,0}$.

 A \textit{local orthogonal pair} (see Definition~\ref{orth}) $\phi$, $\psi$ from $U\subset \mathbb P^{r,s,t}$ to $\mathbb P^{r',s',t'}$ is satisfying $\langle \phi(p),\psi(q)\rangle_{r',s',t'}=0$ for any $\langle p,q\rangle_{r,s,t}=0$ in $U$, which is a generalization of a Segre map between generalized balls. (For more details see \cite{Ga}). When $(f,f)$ is an local orthogonal pair, $f$ is also called a \textit{local orthogonal map} from $\mathbb P^{r,s,t}$ to $\mathbb P^{r',s',t'}$.
In what follows, we call $(\phi, \psi)$ \textit{null} if $\langle 
\phi(p),\psi(q)\rangle_{r',s',t'}=0$ for all $p$, $q\in U$. This is a kind of triviality in the current setting analogous to constant maps. We call $(\phi, \psi)$  \textit{quasi-standard} if it is in some sense a ``direct sum" of two parts, of which one comes from a linear conformal pair from $\mathbb C^{r,s,t}$ to $\mathbb C^{r',s',t'}$ and the other one is null. The exactly definition can be found in Definition \ref{stand}.

 The following result serves as a fundamental tool for analyzing the rigidity properties of orthogonal pairs, particularly in cases where the source domains have Levi-degenerated boundaries. Our approach to the study of holomorphic mapping between generalized balls is to work on these geometric objects directly. While our method might not be able to produce as much detail of the relevant maps as the traditional normal form theory, it has the advantages of being coordinate free and geometrically more transparent. It is especially suited for obtaining rigidities and general behaviors with shorter and easier arguments.

\begin{theorem}\label{degenerate pair}

Fix $r,s,r',s'\in\mathbb Z_{\ge 0}$ such that $r+s\ge 2$. If any local orthogonal pair from $\mathbb P^{r,s}$ to $\mathbb P^{r',s'}$ is quasi-standard or null, then a local orthogonal pair $(\phi,\psi)$ from $\mathbb P^{r,s,t}$ to $\mathbb P^{r',s',t'}$ must be quasi-standard or null for any positive integer $t$ and $t'$.

\end{theorem}

As a corollary, we get the result of orthogonal maps version, which was also in dependently proved  in a very recent paper
\cite{HX2} with a very different  method.

\begin{corollary}\label{map deg}

Fix $r,s,r',s'\in\mathbb Z_{\ge 0}$ such that $r+s\ge 2$. If any local orthogonal map from $\mathbb P^{r,s}$ to $\mathbb P^{r',s'}$ is quasi-standard or null, then a local orthogonal map $f$ from $\mathbb P^{r,s,t}$ to $\mathbb P^{r',s',t'}$ must be quasi-standard or null for any positive integer $t$ and $t'$.

\end{corollary}
We present a new rigidity property concerning orthogonal pairs (see Definition~\ref{orth}), which are generalizations of Segre maps. The investigation of rigidities related to Segre maps can be found in \cite{Zh}. To the best of the author's knowledge, this is the first result addressing the rigidity of orthogonal pairs and hence Segre maps different dimensions in the Levi-degenerate case.

\begin{theorem}
    Let $r,s,t,n'$ be non-negative integers satisfying $ n'\le 2(r+s)-3$.
    Then a local orthogonal pair $F=(\phi,\psi)$ from $\mathbb P^{r,s,t}$ to $\mathbb P^{n',0}$ is either quasi-standard or null.

\end{theorem}

The paper is organized as follows:
Section 2 serves as an introduction to the main tools employed in this paper, namely, orthogonal pairs and orthogonal maps. It elaborates on the fundamental concepts and facts related to these tools. By presenting the definitions and properties of orthogonal pairs and orthogonal maps in this section, we establish a foundation for the subsequent discussions.
Section 3 is dedicated to our primary objective, which is to demonstrate how the rigidity of orthogonal pairs, especially between generalized balls with Levi-degenerate boundaries, can be deduced from their counterparts involving generalized balls with Levi-non-degenerate boundaries. In section 4, we present some new results about Huang's lemma, which are developed based on the rigidity properties of orthogonal pairs elucidated in Section 3. This results extend the existing Huang's lemma.


\section{Orthogonal pair and orthogonal map}

We shall use the setup, notation and conventions in \cite{GN} and \cite{Ga}. For the reader's convenience, we recall them here.
Let $n\geq 1$ and $r,s,t\in\{0,\ldots,n\}$ such that $r+s+t= n$. Consider the (possibly indefinite and degenerate) Hermitian form on $\mathbb C^{n}$ defined by
$$
\langle z,w\rangle_{r,s,t}:=z_1\bar w_1+\cdots+z_r\bar w_r
-z_{r+1}\bar w_{r+1}-\cdots-z_{r+s}\bar w_{r+s}
$$
and write the norm squared as $\|z\|^2_{r,s,t}:=\langle z,z\rangle_{r,s,t}$.
We denote by $\mathbb C^{r,s,t}$ the $\mathbb C^{n}$ with the Hermitian form
defined above and by $\mathbb P^{r,s,t}:=\mathbb P\mathbb C^{r,s,t}$ its projectivization.
If $t=0$, we will simply write $\langle \cdot,\cdot\rangle_{r,s}$ and $\|\cdot\|_{r,s}$.
We write $\mathbb C^{r,s}$ and  $\mathbb P^{r,s}$  instead of $\mathbb C^{r,s,0}$ and $\mathbb P^{r,s,0}$.

The \textit{orthogonal complement} of $z$ is defined as
$$z^{\perp}=\{w\in  \mathbb P^{r,s,t} \mid \langle z, w\rangle_{r,s,t}=0\}.$$
Note that the orthogonality with respect to $\langle\cdot,\cdot\rangle_{r,s,t}$ still makes sense on $\mathbb P^{r,s,t}$.

Similarly, although the norm $\|\cdot\|^2_{r,s,t}$ does not descend to $\mathbb P^{r,s,t}$, the positivity, negativity or nullity of a line (1-dimensional subspace)
in $\mathbb C^{r,s,t}$ is well defined
and thus we can talk about positive points, negative points and null points on  $\mathbb P^{r,s,t}$.

Now let $H$ be a complex linear subspace in $\mathbb C^{r,s,t}$ and the restriction of  $\langle\cdot,\cdot\rangle_{r,s,t}$ on $H$ has the signature $(a; b; c)$. Obviously, we have $0\le a\le r$, $0\le b\le s$, $0\le c\le \min\{r-a,s-b\}+t$ and $a+b+c=\dim(H)$.  Then $\mathbb PH \cong \mathbb P^{a,b,c}$. We call $\mathbb P H$ an\textit{ $(a,b,c)$-subspace} of $\mathbb P^{r,s,t}$. If $c=0$, we will simply write $(a,b)$-subspace

Maps preserving the orthogonality have
been studied in \cite{GN} and \cite{Ga}.
We first recall some definitions:

\begin{definition}\label{orth}
	 Let
$\phi$, $\psi$ be two holomorphic maps from an open set $U\subset \mathbb P^{r,s,t}$ into $\mathbb
P^{r',s',t'}$. We call $F:=(\phi,\psi)$   an  \textbf {orthogonal pair} if
$\phi(z)\perp \psi(w)$ for any $z$,$w\in U$
such that
$z \perp w$. We also call $F$ a local orthogonal pair from  $\mathbb P^{r,s,t}$ to $\mathbb P^{r',s',t'}$. If $(\phi,\phi)$ is a local orthogonal pair from  $\mathbb P^{r,s,t}$ to
$\mathbb P^{r',s',t'}$, then $\phi$ is called a local \textbf{orthogonal map} from  $\mathbb P^{r,s,t}$ to $\mathbb P^{r',s',t'}$.
\end{definition}

In the study of orthogonal pairs/maps, one of the main goals is to determine under what conditions an orthogonal pair/map comes from some \textit{standard maps}, in the following sense:

\begin{definition}[\cite{Ga}]\label{stand}
Let $(\phi,\psi)$ be a local orthogonal pair from $\mathbb
P^{r,s,t}$ to
$\mathbb P^{r',s',t'}$.
We call $(\phi,\psi)$ is
a \textbf{standard pair} if it is induced by a linear conformal pair from $\mathbb C^{r,s,t}$ to $\mathbb C^{r',s',t'}$,
i.e. a pair of linear maps $\Phi,\Psi:\mathbb C^{r,s,t}\rightarrow\mathbb C^{r',s',t'}$ such that for some $\lambda\in\mathbb R^*$, $\langle\Phi(z),\Psi(w)\rangle_{r',s',t'}=\lambda\langle z,w\rangle_{r,s,t}$ for every $z,w\in \mathbb C^{r,s,t}$

We call
$(\phi,\psi)$ a \textbf{null pair} if
$\phi(z)\perp \psi(w)$ for any $z,w$ in the definition domain.

We call $(\phi,\psi)$ \textbf{quasi-standard} if
there exists an orthogonal decomposition $\mathbb C^{r',s',t'}=V\oplus W$ such that one of the following scenarios occur:

(1) The image of $\phi$ or $\psi$ is contained in $\mathbb PV$ and
$(\pi^V\circ\phi, \pi^V\circ \psi)$
is standard;

(2) The pair $(\pi^V\circ\phi, \pi^V\circ \psi)$ is standard and
$(\pi^W\circ\phi, \pi^W\circ \psi)$
is  null, where $\pi^V$ and $\pi^W$ are the projections from $\mathbb P^{r',s',t'}$ to $\mathbb PV$ and $\mathbb PW$
induced by the orthogonal projections $\mathbb C^{r',s'.t'}\rightarrow V$ and $\mathbb C^{r',s'.t'}\rightarrow W$.

Similarly, we can define standard orthogonal maps, null orthogonal maps, etc.
\end{definition}

Now, we explain the relationship between Huang's equations and orthogonal pairs.
Let
$\{\phi_j\}_{j=1}^{r'+s'}$ and $\{\psi_j\}_{j=1}^{r'+s'}$
be holomorphic functions in a neighborhood of $z=0\in\mathbb C^{r,s,t}$.
Suppose there exists a complex-valued
real-analytic function $H(z,\bar z)$ such that the following equation is satisfied for $z\approx 0$,
\begin{equation}\label{S1}
\sum_{j=1}^{r'}\phi_j(z)\overline{\psi_j(z)}-\sum_{j=r'+1}^{r'+s'}\phi_j(z)\overline{\psi_j(z)}=\|z\|_{r,s,t}^2H(z,\bar z),
\end{equation}
then $\{\phi_j\}$ and $\{\psi_j\}$ can induce a local orthogonal pair after homogenization, as follows.

Let $n:=r+s+t$.
Homogenizing $\{\phi_j(z)\}$ and $\{\psi_j(z)\}$ by using the substitution $z_k=x_k/x_{n+1}$, for $1\leq k\leq n$ and multiplying the equation by $|x_{n+1}|^{2}$, then we get an equation of bihomogeneous functions
of degree 1 in $x=[x_1,\ldots,x_{n+1}]$:

\begin{equation}
\sum_{j=1}^{r'}\tilde\phi_j\left(x\right)\overline{\tilde\psi_j\left(x\right)}-\sum_{j=r'+1}^{r'+s'}\tilde\phi_j\left(x\right)\overline{\tilde\psi_j\left(x\right)}
=\|x\|_{r,s,t+1}^2\tilde H\left(x,\bar x\right),
\end{equation}
where $\tilde{\phi}_j(x)=x_{n+1}\phi_j\left(\dfrac{x_1}{x_{n+1}},\ldots,\dfrac{x_n}{x_{n+1}}\right)$, $\tilde{\psi}_j(x)=x_{n+1}\psi_j\left(\dfrac{x_1}{x_{n+1}},\ldots,\dfrac{x_n}{x_{n+1}}\right)$ and $\tilde{H}(x,\bar x)=H\left(\dfrac{x_1}{x_{n+1}},\ldots,\dfrac{x_n}{x_{n+1}},\overline{\left(\dfrac{x_1}{x_{n+1}}\right)},\ldots,\overline{\left(\dfrac{x_n}{x_{n+1}}\right)}\right)$.

Polarizing the equation above,  we get
\begin{equation}\label{polarized eq}
\sum_{i=1}^{r'}\tilde{\phi}_i(x)\overline{\tilde{\psi}_i(y)}-\sum_{j=r'+1}^{r'+s'}\tilde{\phi}_j(x)\overline{\tilde{\psi}_j(y)}=\langle x,y\rangle_{r,s,t+1}\tilde{H}(x,\bar y).
\end{equation}
Let $ \widetilde{\phi}=[\tilde{\phi}_1\cdots, \tilde{\phi}_{r\rq{}+s\rq{}}]$ and $ \widetilde{\psi}=[\tilde{\psi}_1\cdots, \tilde{\psi}_{r\rq{}+s\rq{}}]$. Then we deduce that $\langle \widetilde{\phi}(x),\widetilde{\psi}(y)\rangle_{r',s'}=0$ for every $x,y\in \mathbb P^{r,s,t+1}$ such that $\langle x,y\rangle_{r,s,t+1}=0$. Therefore, $F=(\widetilde{\phi}, \widetilde{\psi})$ is a local orthogonal pair from $\mathbb P^{r,s,t+1}$ to $\mathbb P^{r',s'}$.

Therefore, we can now study the holomorphic functions $\{\phi_i\}$ and $\{\psi_i\}$ satisfying equation (\ref{S1}) through studying the orthogonal pairs from $\mathbb P^{r,s,t+1}$ to $\mathbb P^{r',s'}$ induced by $\{\phi_i\}$ and $\{\psi_i\}$. Hence, we can use techniques and results for orthogonal pairs.


\section{Rigidity of orthogonal maps and pairs}

In this section, we will generalize some rigidity theorems obtained in~\cite{GN} and~\cite{Ga} for orthogonal maps and pairs. These generalizations are made to incorporate  degeneracy in the Hermitian forms of the domain space. As seen in the previous section, such degeneracy necessarily appears when we homogenize the functions satisfying a Huang's lemma-type equation (e.g. in Theorem~\ref{pair}). The degeneracy of the Hermitian form in the source space will give us new difficulty in obtaining rigidity theorems, mainly due to the fact that a general hyperplane is no longer an orthogonal complement (of any point). We will use other method to handle these cases.

We have seen that a Huang's lemma-type equation will give rise to a local orthogonal pair. The following proposition is a converse in some sense.

\begin{proposition}\label{multiplier prop}
Let $(\phi,\psi):U\subset\mathbb P^{r,s,t}\rightarrow\mathbb P^{r',s',t'}$ be a local orthogonal pair, where $U$ is a connected non-empty open set. There exist an open set $V\subset U$ and holomorphic functions $\{\phi_j\}_{j=1}^{r'+s'+t'}$ and  $\{\psi_j\}_{j=1}^{r'+s'+t'}$ of $z=(z_1,\ldots, z_{r+s+t})\in\widetilde V$, where $\widetilde V$ is the preimage of $V$ with respect to the projection $\mathbb C^{r,s,t}\setminus\{0\}\rightarrow\mathbb P^{r,s,t}$, such that
$$
\phi([z])=[\phi_1(z),\ldots,\phi_{r'+s'+t'}(z)], \,\,\psi([z])=[\psi_1(z),\ldots,\psi_{r'+s'+t'}(z)]
$$
for $z\in\widetilde V$ and each $\phi_j$,$\psi_j$, as a function of the variables $(z_1,\ldots,z_{r+s+t})$, is homogeneous of degree $1$. Furthermore, there exists a complex-valued real analytic function $H(z,\bar z)$ on $\widetilde V$ such that for $z\in\widetilde V$,
\begin{equation}\label{H}
\sum_{j=1}^{r'}\phi_j(z)\overline{\psi_j(z)}-\sum_{j=r'+1}^{r'+s'}\phi_j(z)\overline{\psi_j(z)}=\|z\|_{r,s,t}^2H(z,\bar z).
\end{equation}

\end{proposition}

\begin{proof}
If $(\phi,\psi)$ is null, then for all $p,q\in U$,  $\langle\phi(p),\psi(q)\rangle_{r',s',t'}=0$. Therefore, $
\sum_{j=1}^{r'}\phi_j(z)\overline{\psi_j(z)}-\sum_{j=r'+1}^{r'+s'}\phi_j(z)\overline{\psi_j(z)}\equiv 0.$ The proposition has been proved.

When $(\phi,\psi)$ is not null, without loss of generality, and by shrinking the domain of definition of $(\phi,\psi)$ to a sufficiently small open subset $V\subset U$, we may assume that $V$ (resp. $\phi(V)$ and $\psi(V)$) is contained in the Euclidean coordinate chart in $W\subset\mathbb P^{r,s,t}$ defined by $z_1\neq 0$ (resp. $W'\subset\mathbb P^{r',s',t'}$ defined by $z'_1\neq 0$). Write the inhomogeneous coordinates in $W$ as $(w_2,\ldots, w_{r+s+t})$, where $w_k=z_k/z_1$ for $2\leq k\leq r+s+t$.

In terms of these inhomogeneous coordinates, the map $\phi$ and $\psi$ can be represented by $[1,w_2,\ldots,w_{r+s+t}]\mapsto [1,\phi^\sharp_2(w),\ldots,\phi^\sharp_{r'+s'+t'}(w)]$ and $[1,w_2,\ldots,w_{r+s+t}]\mapsto [1,\psi^\sharp_2(w),\ldots,\psi^\sharp_{r'+s'+t'}(w)]$ respectively, where each $\phi^\sharp_j$ and $\psi^\sharp_j$ are  holomorphic functions of $w=(w_2,\ldots,w_{r+s+t})$. Furthermore, the set of null points in $W$ is defined by the equation $1+\sum_{k=2}^r|w_k|^2-\sum_{k=r+1}^{r+s}|w_k|^2=0$ . As $(\phi,\psi)$ preserves the orthogonalities (which follows directly from the definition of orthogonal pairs), by shrinking $V$ if necessary, we deduce that there exists a complex-valued real analytic function $H^\sharp(w,\bar w)$ on $V$ such that
$$
1+\sum_{j=2}^{r'}\phi^\sharp_j(w)\overline{\psi^\sharp_j(w)}-\sum_{j=r'+1}^{r'+s'}\phi^\sharp_j(w)\overline{\psi^\sharp_j(w)}
=\left(1+\sum_{k=2}^r|w_k|^2-\sum_{k=r+1}^{r+s}|w_k|^2\right)H^\sharp(w,\bar w).
$$

Rewrite the equation in terms of the variables $(z_1,\ldots, z_{r+s+t})$ and multiply both sides by $|z_1|^2$, we get the desired equation
$$
\sum_{j=1}^{r'}\phi_j(z)\overline{\psi_j(z)}-\sum_{j=r'+1}^{r'+s'}\phi_j(z)\overline{\psi_j(z)}=\|z\|_{r,s,t}^2H(z,\bar z),
$$
where $\phi_1:=z_1$,$\psi_1:=z_1$ $\phi_j:=z_1\phi^\sharp_j(z_2/z_1,\ldots, z_{r+s+t}/z_1)$, $\psi_j:=z_1\phi^\sharp_j(z_2/z_1,\ldots, z_{r+s+t}/z_1)$  for $2\leq j\leq r'+s'+t'$; and $H(z,\bar z):=H^\sharp(z_2/z_1,\ldots, z_{r+s+t}/z_1)$.
\end{proof}

Proposition \ref{multiplier prop} showes the properties of $H(z,\bar z)$ is very important during the study of orthogonal pairs and Huang's lemma.

\begin{definition}\label{multiplier}
For a given a local orthogonal pair $(\phi,\psi)$ from $\mathbb P^{r,s,t}$ to $\mathbb P^{r',s',t'}$, the complex-valued real analytic function $H(z,\bar z)$ satisfying the following equation
\begin{equation*}
\sum_{j=1}^{r'}\phi_j(z)\overline{\psi_j(z)}-\sum_{j=r'+1}^{r'+s'}\phi_j(z)\overline{\psi_j(z)}=\|z\|_{r,s,t}^2H(z,\bar z).
\end{equation*}
is called the multiplier of $(\phi,\psi)$. The multiplier is uniquely determined up to a factor of the type $h_1(z)\overline{h_2(z)}$, where $h_1(z)$, $h_2(z)$ are holomorphic functions.  
\end{definition}

It is clear that for a given local orthogonal pair $(\phi,\phi)$, then $H(z,\bar z)$ is a real-valued analytic function, which is aslo called the \textit{multiplier} of orthogonal map $\phi$ and is determined up to some positive factor. The multiplier is essential for the properties of $\phi$, as we will see below. See~\cite{GN1} for a  study of the multipliers for rational proper maps between generalized balls.

  Proposition \ref{multiplier prop} shows the relationship between orthogonal pairs and Huang's type equations. Propostion \ref{map rank} gives the relationship between the  multiplier and rigidities of orthogonal pairs between projective spaces. Therefore, we can get some rigidities of orthogonal pairs from the information about the multiplier.

\begin{proposition}\label{map rank}
    Let $(\phi,\psi)$ be an orthogonal pair from $\mathbb P^{r,s,t}$ to $\mathbb P^{r',s',t'}$ and $H(z,\bar z)$ be the multiplier associated to $(\phi,\psi)$ as in Proposition~\ref{multiplier prop}. Then, $(\phi,\psi)$ is null if and only if $H(z,\bar z)\equiv 0$.   
   Additionally, $(\phi,\psi)$ is quasi-standard if and only if there exists two non-zero holomorphic functions $h_1(z)$ and $h_2(z)$ such that $H(z,\bar z)=h_1(z)\overline{h_2(z)}$.

\end{proposition}
\begin{proof}

The result concerning the null case is a direct consequence of the definition of a null orthogonal map.

By Proposition \ref{multiplier prop}, write $\phi=[\phi_1(z),\cdots,\phi_{r'+s'+t'}(z)]$ and $\psi=[\psi_1(z),\cdots,\psi_{r'+s'+t'}(z)]$.
Then we have $$\sum_{j=1}^{r'}\phi_j(z)\overline{\psi_j(z)}-\sum_{j=r'+1}^{r'+s'}\phi_j(z)\overline{\psi_j(z)}=\|z\|_{r,s,t}^2H(z,\bar z).$$

We present a proof for the "if" part as the "only if" statement is obvious. Assume that there exists two non-zero holomorphic functions $h_1(z)$ and $h_2(z)$ such that $H(z,\bar z)=h_1(z)\overline{h_2(z)}\not\equiv 0$. So $(\phi,\psi)$ is not null.

Let $\pi$ be the canonical projection from $\mathbb P^{r',s',t'}$ to $\mathbb P^{r',s'}$ as a rational map. Obviously,
$\langle\phi,\psi\rangle_{r',s',t'}=\langle\pi\circ\phi,\pi\circ\psi\rangle_{r',s'}$. In addition, $(\phi,\psi)$ is quasi-standard if and only if $(\pi\circ\phi,\pi\circ\psi)$ is quasi-standard.
  Therefore, we only need to prove the statement for $t'=0$.

Let's assume that $t'=0$ from now on, i.e $(\phi,\psi)$ is a local orthogonal pair from $\mathbb P^{r,s,t}$ to $\mathbb P^{r',s'}$ . Define the map $\Psi=[\psi_1(z),\cdots,\psi_{r'}(z),$ $-\psi_{r'+1}(z),\cdots,-\psi_{r'+s'}(z)]$.  Then $(\phi, \Psi)$ is a local orthogonal pair from $\mathbb P^{r,s,t}$ to $\mathbb P^{r'+s',0}$.  It is worth noting that $(\phi, \psi)$ is quasi-linear (repectively, null) if and only if  $(\phi, \Psi)$ is quasi-linear (repectively, null).
Denote the linear span of $\phi(U)$ by $K$ with $k=\dim K$. Since the Hermitian form in $\mathbb P^{r'+s',0}$ is positive, then composition  $(\pi_K\circ\phi, \pi_K\circ\Psi)$ is a local orthogonal pair from $\mathbb P^{r,s,t}$ to $K\cong \mathbb P^{k+1,0}$. On the other hand, $(\pi_K^{\perp}\circ\phi, \pi_K^{\perp}\circ\Psi)$ is null where $\pi_K$ and $\pi_K^{\perp}$ are the projection from $\mathbb P^{r'+s',0}$
to $K$ and $K^{\perp}$, respectively. By Proposition \ref{multiplier prop}, write $\pi_K\circ\phi=[\phi'_1(z),\cdots,\phi'_{k+1}(z)]$ and $\pi_K\circ\Psi=[\psi'_1(z),\cdots,\psi'_{k+1}(z)]$. Then 

$$\sum_{j=1}^{k+1}\phi'_j(z)\overline{\psi'_j(z)}=\sum_{j=1}^{r'}\phi_j(z)\overline{\psi_j(z)}-\sum_{j=r'+1}^{r'+s'}\phi_j(z)\overline{\psi_j(z)}=\|z\|_{r,s,t}h_1(z)\overline{h_2(z)}.$$

Polarizing the above equation, we get 
 $$\sum_{j=1}^{k+1}\frac{\phi'_j(z)}{h_1(z)}\frac{\overline{\psi'_j(w)}}{\overline{h_2(w)}}=\langle z,w\rangle_{r,s,t}.$$

As the linear span $\pi_K\circ\phi(U)$ is $K$, we can choose $(k+1)$ points $\{p_i\}_{i=1}^{k+1}$ in  $U$ to get $(k+1)$-linearly independent equations.
\begin{equation}
\begin{cases}
	\sum_{j=1}^{k+1}\frac{\phi'_j(p_1)}{h_1(p_1)}\frac{\overline{\psi'_j(w)}}{\overline{h_2(w)}}=\langle p_1,w\rangle_{r,s,t}\\
	\cdots\\
	\sum_{j=1}^{k+1}\frac{\phi'_j(p_{k+1})}{h_1(p_{k+1})}\frac{\overline{\psi'_j(w)}}{\overline{h_2(w)}}=\langle p_{k+1},w\rangle_{r,s,t}\\
\end{cases}	
\end{equation}

 We can solve these equations and obtain each  $\frac{\overline{\psi'_j(w)}}{\overline{h_2(w)}}$ is linear for $1\le j\le k+1$.Therefore, $\pi_K\circ\Psi$ is a linear map.
 
 Denote the linear span of $\pi\circ\Psi(U)$ by $L$. 
 Similarly, $(\Pi_L\circ\pi_K\circ\phi, \Pi_L\circ\pi_K\circ\Psi)$ is a local orthogonal map from $\mathbb P^{r,s,t}$ to $L$ and $(\Pi_L^{\perp}\circ\pi_K^{\perp}\circ\phi, \Pi_L^{\perp}\circ\pi_K^{\perp}\circ\Psi)$ is null where  $\Pi_L$ and $\Pi_L^{\perp}$ are the projection from $K$ to $L$ and $L^{\perp}$, respectively. Using the same method, we can solve some equations and get that $\Pi_L\circ\pi_K\circ\phi$ is a linear map. It is not difficult to prove $(\Pi_L\circ\pi_K\circ\phi, \Pi_L\circ\pi_K\circ\Psi)$ is a standard pair when $\Pi_L\circ\pi_K\circ\phi$, $\Pi_L\circ\pi_K\circ\Psi$ are linear maps and $(\Pi_L\circ\pi_K\circ\phi, \Pi_L\circ\pi_K\circ\Psi)$ is not null. 
 Therefore, $(\phi,\psi)$ is quasi-standard pair.
\end{proof}

Since $(\phi,\phi)$ is a local orthogonal pair if and only if $\phi$ is a local orthogonal map, it is very natural to study orthogonal maps by analyzing 
orthogonal pairs. 

\begin{corollary}\label{multiplier map}
 Let $\phi$ be an orthogonal map from $\mathbb P^{r,s,t}$ to $\mathbb P^{r',s',t'}$ and $H(z,\bar z)$ be the multiplier associated to $\phi$. Then, $\phi$ is null if and only if $H(z,\bar z)\equiv 0$.   
   Additionally, $\phi$ is quasi-standard if and only if there exists a non-zero holomorphic function $h(z)$ such that $H(z,\bar z)=\|h(z)\|^2$.
\end{corollary}

From Proposition \ref{map rank}, we can get the rigidities of orthogonal pairs by studying the multipliers. 
 
 \begin{lemma}\label{restriction}
 Let $F(z,\bar z)$ be a complex-valued real analytic function over $\mathbb C^n$. If for any general hyperplane $\Xi$ in $\mathbb C^n$, there exists two holomorphic functions $f_\Xi(z)$,$g_\Xi(z)$ such that $F(z,\bar z)|_\Xi=f_\Xi(z)\overline{g_\Xi(z)}$, then there exist two holomorphic functions $f(z)$,$g(z)$ such that  $F(z,\bar z)=f(z)\overline{g(z)}$.
 \end{lemma}
 
 \begin{proof}
 Since $F(z,\bar z)$ is real analytic function, $F(z,\bar z)|_\Xi $ is also real analytic on $\Xi$. Since $F(z,\bar z)|_\Xi=f_\Xi(z)\overline{g_\Xi(z)}$ for any hyperplane $\Xi\subset \mathbb C^n$, then $\partial\overline{\partial}\log F(z,\bar z)|_\Xi \equiv 0 $. Hence, $\partial\overline{\partial}\log F(z,\bar z)\equiv 0 $, which means $\log F(z,\bar z)$ is a sum of a holomorphic function and an anti-holomorphic function. Hence there exist two holomorphic functions $f(z)$, $g(z)$ such that  $F(z,\bar z)=f(z)\overline{g(z)}$.

 \end{proof}

 Theorem \ref{degenerate pair} explores that rigidity of orthogonal pairs between generalized balls with Levi-degenerate boundaries can be deduced from their counterparts involving generalized balls with Levi-non-degenerate boundaries.

\begin{proof}[Proof of Theorem \ref{degenerate pair}]:

Let $\pi$ be the projection from $\mathbb P^{r',s',t'}$ to $\mathbb P^{r',s'}$ as
a rational map. Then $(\pi\circ\phi,\pi\circ\psi)$ is an local orthogonal map from $\mathbb P^{r,s,t}$ to $\mathbb P^{r',s'}$. In addition, $(\phi,\psi)$ is quasi-standard (respectively, null) if and only if $(\pi\circ\phi,\pi\circ\psi)$ is quasi-standard(respectively, null). Therefore, we only need to prove the results for $t'=0$.

Let $(\phi,\psi)$ be a local orthogonal pair from $\mathbb P^{r,s,t}$ to $\mathbb P^{r',s'}$.
By Proposition \ref{multiplier prop},
write $\phi:=[\phi_1,\cdots,\phi_{r'+s'}]$  and $\phi:=[\psi_1,\cdots,\psi_{r'+s'}]$ and let  $H(z,\bar z)$ be the multiplier of $\phi$, i.e.
    $$
    \sum_{i=1}^{r'}\phi_i(z)\overline{\psi_i(z)}-\sum_{i=r'+1}^{r'+s'}\phi_i(z)\overline{\psi_i(z)}=H(z,\bar z)\|z\|_{r,s,t}^2.
    $$

 Denote by $(\phi^{V},\psi^{V})$ the restriction of $(\phi,\psi)$ on an $(r,s)$-subspace $V$. Note that any general $(r+s-1)$-plane is an $(r,s)$-subspace in $\mathbb P^{r,s}$. Then $(\phi^{V},\psi^{V})$ is a local orthogonal pair from $\mathbb P^{r,s}$ to $\mathbb P^{r',s'}$.

By assumption, any orthogonal pair from $\mathbb P^{r,s}$ to $\mathbb P^{r',s'}$ is quasi-standard or null, which implies $(\phi^{V},\psi^{V})$ is quasi-standard or null. If $(\phi^{V},\psi^{V})$ is null for any choice of $V$, then from Proposition \ref{map rank}, $H(z,\bar z)|_V\equiv 0$ for any an $(r,s)$-subspace $V$ in $\mathbb P^{r,s,t}$. Hence, $H(z,\bar z)\equiv 0$. So $(\phi,\psi)$ is null.
If $(\phi^{V},\psi^{V})$ is quasi-standard for some $V$, then for a general choice of $V$, it is quasi-standard. Therefore, from Proposition \ref{map rank}, there exist two holomorphic functions $f_V(z)$ and $g_V(z)$ such that  $H(z,\bar z)|_V=f_V(z)\overline{g_V(z)}$ for any $(r,s)$-subspace $V$ in $\mathbb P^{r,s,t}$. Using Lemma \ref{restriction} several times, there exist two holomorphic functions $f(z)$ and $g(z)$ such that  $H(z,\bar z)=f(z)\overline{g(z)}$. Due to Proposition \ref{map rank}, $(\phi,\psi)$ is quasi-standard.

\end{proof}

The majority of existing results on the rigidity properties of proper holomorphic maps primarily concentrate on cases where the source domains are generalized balls with Levi-non-degenerate boundaries (See\cite{BEH},\cite{BH}, \cite{HLTX},etc). The version of Theorem \ref{degenerate pair} for orthogonal map is Corollary \ref{map deg}, which is the main theorem in \cite{HX2}. Actually, Corollary 
\ref{map deg} can be deduced from Theorem \ref{degenerate pair}. We also give an alternative simple proof here.

\begin{proof}[Proof of Corollary \ref{map deg}]:

Similar to the proof of Theorem \ref{degenerate pair}, it suffices to prove the statement for the case when $t'=0$.
By Proposition \ref{multiplier prop},
write a local orthogonal map as $\phi:=[\phi_1,\cdots,\phi_{r'+s'}]$ from $\mathbb P^{r,s,t}$ to $\mathbb P^{r',s'}$ and let  $H(z,\bar z)$ be the multiplier of $\phi$, i.e.
    $$
    \sum_{i=1}^{r'}|\phi_i(z)|^2-\sum_{i=r'+1}^{r'+s'}|\phi_i(z)|^2=H(z,\bar z)\|z\|_{r,s,t}^2.
    $$

    The multiplier $H(z,\bar z)$ has a decomposition of the following form:
    $$H(z,\bar z)=|a_1(z)|^2+\cdots+|a_p(z)|^2-|b_1(z)|^2-\cdots-|b_q(z)|^2,$$
    where $a_i$, $b_j$ are homogeous holomorphic functions with the same degree and $\mathbb C$-linearly independent (See \cite{Da}, P33). And $p$, $q$ can be infinity. Then $H(z,\bar z)$ defines a holomorphic map $\mathcal A$ from $\mathbb P^{r,s,t}$ to $\mathbb P^{p+q-1}$ with $\mathcal A(z)=[a_1(z),\cdots,a_p(z)$, $b_1(z),\cdots, b_q(z)]$.

 Denote by $\phi^{V}$ the restriction of $\phi$ on an $(r,s)$-subspace $V$. Any general $(r+s-1)$-plane is an $(r,s)$-subspace in $\mathbb P^{r,s}$. Then $\phi^{V}$ is a local orthogonal map from $\mathbb P^{r,s}$ to $\mathbb P^{r',s'}$.  By assumption, $\phi^{V}$ is quasi-standard or null. From Corollary \ref{multiplier map}, the multiplier of $\phi^{V}$ is either $0$ or rank $1$. Then either $H(z,\bar z)|_{V}=0$ or the holomorphic map $\mathcal A|_V$ is a constant map.

 When $H(z,\bar z)\mid_V=0$, then $H(z,\bar z)\equiv 0$. Hence $\phi$ is null.
 
When $\mathcal A|_V$ is a constant map for a general choice of $V$, then it follows that $\mathcal A$ is a constant map, i.e. $H(z,\bar z)$ has rank $1$. Then $\phi$ is quasi-standard from Corollary \ref{multiplier map}.

\end{proof}

In order to give some new rigidities of orthogonal pairs, we recall some results about the rigidity of orthogonal pairs and rewrite them in \cite{Ga} for our convenience.

\begin{theorem}\label{Segre}
	(Theorem 1.3 \cite{Ga})
	Let $r,s,n'$ be non-negative integers
	and  $F$ be a local orthogonal pair from $\mathbb P^{r,s}$ to $\mathbb P^{n',0}$. If $n'\le 2(r+s)-3$, $F$ is quasi-standard or null.
	
\end{theorem}

The following theorem about rigidities of orthogonal pairs follows directly from the combination of Theorem \ref{Segre} and Theorem \ref{degenerate pair}.
\begin{theorem}\label{main}
Let $r,s,t,n'$ be non-negative integers such that $ n'\le 2(r+s)-3$.
A local orthogonal pair $F=(\phi,\psi)$ from $\mathbb P^{r,s,t}$ to $\mathbb P^{n',0}$ is quasi-standard or null.

\end{theorem}

\section{Proof of Theorem \ref{pair}}

In this section, we present some results that generalize Huang's Lemma by analyzing orthogonal pairs. The following proposition serves as a bridge between Huang's Lemma and rigidities of orthogonal pairs between projective spaces.

\begin{proposition}\label{pair rank}
For $r,s,t,n'\in \mathbb Z_{\ge 0}$ such that $n'\ge {r+s}\ge 2$,
        let $(\phi,\psi)$ be a local orthogonal pair from $\mathbb P^{r,s,t}$ to $\mathbb P^{n',0}$. 
         If $(\phi,\psi)$ is quasi-standard, then there exist holomorphic
functions, $h_1(z)$,$h_2(z)$ $\{\Phi_i\}^q_{i=1}$, $\{\Psi_i\}^q_{i=1}$, where
$q=n'-r-s$ and  matrices $B$, $C$ satisfying $\overline
{C^T}B=H_{r,s,q}$ such that
$$
[\phi_1,\cdots,\phi_{n'}]^T=B[h_1(z)z_1,\cdots, h_1(z)z_{r+s},\Phi_1,\cdots,\Phi_q]^T,
$$
$$
[\psi_1,\cdots,\psi_{n'}]^T=C[h_2(z)z_1,\cdots,  h_2(z)z_{r+s},\Psi_1,\cdots,\Psi_q]^T.
$$

\end{proposition}
\begin{proof}
    From the definition of quasi-standard pair, there exists an orthogonal decompostion decomposition $\mathbb C^{n',0}=V\oplus V^{\perp}$ such that $(\pi\circ \phi,\pi\circ \psi)$ is a standard pair and $(\pi^{\perp}\circ \phi,\pi^{\perp}\circ \psi)$ is null, in which $\pi$ and $\pi^{\perp}$ are projections from $\mathbb P^{n',0}$ to $\mathbb PV$ and $\mathbb PV^{\perp}$, respectively. After composing a standard orthogonal pair,   $\pi\circ\phi=[z_1,\cdots,z_r,z_{r+1},\cdots,z_{r+s}]$ and $\pi\circ\psi=[z_1,\cdots,z_r,-z_{r+1},\cdots,-z_{r+s}]$.
   Moreover, $\langle \pi^{\perp}\circ\phi(z),\pi^{\perp}\circ\psi(z) \rangle\equiv 0$
for any $z\in U$.
    Let $n'-r-s=q$.
     It follows that there exist holomorphic functions $\Phi_i$, $\Psi_i$ and two matrices $B$, $C$ satisfying $\bar {C^T}B=H_{r,s,q}$ such that
$$
[\phi_1,\cdots,\phi_{n'}]^t=B[h_1(z)z_1,\cdots,h_1(z)z_{r+s},\Phi_1,\cdots,\Phi_q]^t,
$$
$$
[\psi_1,\cdots,\psi_{n'}]^t=C[h_2(z)z_1,\cdots, h_2(z)z_{r+s},\Psi_1,\cdots,\Psi_q]^t.
$$

 \end{proof}

By combining the rigidity properties of orthogonal pairs discussed in Section 3, we are now able to establish the proof for our main Theorem \ref{pair}.

\begin{proof}[Proof of Theorem \ref{pair}]
Any holomorphic functions  $\{\phi_j\}$, $\{\psi_j\}$ and $H(z,\bar z)$ satisfying
$$\sum_{j=1}^{n'}\phi_j(z)\overline{\psi_j(z)}=\|z\|_{r,s}^2H(z,\bar z),$$
can induce a local orthogonal pair from $\mathbb P^{r,s,1}$ to $\mathbb P^{n',0}$ after homogenization, as argued in Section 2.
Then, the result is a direct consequence of Theorem \ref{Segre} and Proposition \ref{pair rank}.
\end{proof}

From Theorem \ref{pair}, it is easy to get the following corollary.

\begin{corollary}\label{map}
    Let $r,s,r',s'\in\mathbb N$
    and $\{\phi_j\}_{j=1}^{r'+s'}$ be holomorphic  functions in $ z\in \mathbb C^{r+s+t}$ near the origin. Assume that $\phi_j(0)=0$ for every $j$.
    Let $H(z,\bar z)$ be a real-analytic function for $z \approx 0$ such that
    $$\sum_{j=1}^{r'}|\phi_j(z)|^2-\sum_{j=r'+1}^{r'+s'}|\phi_j(z)|^2=\|z\|_{r,s,t}^2H(z,\bar z).$$
 
   If $r'+s'\le 2(r+s)-3$, then either $H(z,\bar z)\equiv 0$ or  rank $1$.

Furthermore, when $H(z,\bar z)\not\equiv 0$,
then 
there exist holomorphic functions $h(z)$, $\Phi_1$, $\cdots$, $\Phi_{q}$ where $q:=r'+s'-r-s\geq 0$ and $W\in M(r'+s',r'+s'; \mathbb C)$ satisfying
$\overline W^TH_{r',s'}W=H_{r,s,q}$
such that
$$[\phi_1,\cdots,\phi_{r'+s'}]^T=W
[
h(z)z_1,\cdots,h(z)z_{r+s},\Phi_1,\cdots,\Phi_{q}
]^T.
$$
\end{corollary}

{\bf Remark.} 
Using Corollary \ref{map}, we can deduce the original form of Huang's Lemma and Corollary 1.1 in \cite{Xi}. Furthermore, through the combination of Proposition \ref{map rank}, Theorem \ref{degenerate pair} and the rigidity properties established in \cite{GN} and \cite{HLTX}, etc. for orthogonal mappings and proper mappings between generalized balls, we can directly obtain various formulations of Huang's Lemma, including Lemma 2.1 in [HLTX], Lemma 2.2 in [HJY2], and Lemma 2.2 in [HX], etc.


\begin{thebibliography}{999}
	
\bibitem[Al]{Al} H.	Alexander,  Proper holomorphic maps in $\mathbb C^n$. Indiana Univ. Math. J. {\bf 26}, 137-146 (1977)

\bibitem[BEH]{BEH}M. Baouendi,  P.  Ebenfelt and X. Huang: Super-rigidity for CR embeddings of real hypersurfaces into hyperquadrics, Adv. Math. {\bf 219}, 1427-1445, (2008)



    \bibitem[BH]{BH} M. Baouendi and X. Huang, Super-rigidity for holomorphic mappings between hyperquadics with positive signature, J. Diff. Geom. {\bf 69}, no. 2, 379-398 (2005).

    \bibitem[CS]{CS} J. Cima and T. Suridge, Boundary behavior of rational proper maps, Duke Math. J. {\bf 60} , no. 1, 135-138 (1990).

    \bibitem[Da]{Da} J. D'Angelo, Rational sphere map, Birkhauser, Vol.341.

    \bibitem[DX]{DX} J. D'Angelo and M. Xiao, Symmetries in CR complexity theory, Adv. Math. {\bf 313,} 590-627 (2017).

\bibitem[Eb]{Eb} P.  Ebenfelt, On the HJY Gap Conjecture in CR geometry vs. the SOS Conjecture for polynomials. {\it Analysis and Geometry in Several Complex Variables, Contemp.Math.}, {\bf 681}(2017), 125-135, Providence, Amer. Math. Soc.


    \bibitem[EHZ]{EHZ}P. Ebenfelt, X. Huang and D. Zaitsev: The equivalence problem and rigidity for hypersurfaces embedded into hyperquadrics,  Amer. J.  Math. {\bf 127}, No. 1, 169-191, (2005)

    \bibitem[ES]{ES}P. Ebenfelt and R. Shroff: Partial rigidity of CR embeddings of real hypersurfaces into hyperquadrics with small signature difference, Comm. in Anal.  Geom., {\bf 23}, no. 1,159-190, 2015

\bibitem[Fr]{Fr} F. Forstneri\v{c}, {Extending proper holomorphic mappings of positive codimension}, Invent. Math. {\bf 95}, 31-62 (1989).






\bibitem[GN]{GN}Y. Gao and S.-C. Ng, Local orthogonal maps and rigidity of holomorphic mappings between real hyperquadrics, To appear in J. de Math. Pures et Appl. {\it arXiv.} {\bf 2110.04046}



\bibitem[GN1]{GN1} Y. Gao and S.-C. Ng,{ On rational proper mappings among generalized complex balls}, Asian J. Math.
Vol. {\bf 22}, No. 2, 355-380,  (2018)




\bibitem[Ga]{Ga}Y. Gao, {Orthogonal pair and a rigidity problem for segre maps associated to real hyperquadrics}, J.  Geom.  Analysis, (2022) https://doi.org/10.1007/s12220-022-00949-5



\bibitem[HLTX]{HLTX}  X. Huang, Jin Lu, X. Tang and M. Xiao, {Proper mappings between indefinite hyperbolic spaces and type I classical domains},. Trans. Am. Math. Soc. 375, 8465-8481 (2022).


\bibitem[HJY]{HJY} X. Huang, S. Ji and W. Yin, {Recent progress on two problems in several complex variables}, Proceedings of the ICCM 2007, International Press, Vol I, 563575 (2009).

\bibitem[HJY2]{HJY2} X. Huang, S. Ji and W. Yin, On the third gap for proper holomorphic maps between balls. Math. Ann. 358 (2014), 115-142

\bibitem[Hu]{Hu} X. Huang, {On a linearity problem of proper holomorphic maps between balls in complex spaces of different dimensions}, J. Differential Geom., {\bf 51} (1999) 13-33.

\bibitem[HX]{HX} X. Huang and M. Xiao, { Holomorphic mappings between hyperquadrics with positive signature}, Pure and Applied Math. Quarterly, Vol {\bf 18}, No.2, 599-616, (2022).

\bibitem[HX2]{HX2}  X. Huang, and M. Xiao, Rigidity of mappings between degenerate and indefinite hyperbolic spaces,. J Geom Anal 33, 23 (2023). https://doi.org/10.1007/s12220-022-01080-1


\bibitem[Po]{Po} H. Poincarè, Les fonctions analytiques de deux variables et la représentation conforme. Rend. Circ. Mat. Palermo 23, 185-220 (1907)

\bibitem[NZ]{NZ} S.-C. Ng, and Y. Zhu, Rigidity of Proper Holomorphic Maps Among Generalized Balls with Levi-Degenerate Boundaries. J Geom Anal {\bf 31}, 11702-11713 (2021)

\bibitem[St]{St} B. Stensones,{  Proper maps which are Lipschitz up to the boundary}, J. Geom. Anal. 6 {\bf no. 2}, 317-339 (1996).





\bibitem[We]{We}  S.M., Webster: Pseudohermitian structures on a real hypersurface. J. Differential Geom. 13, 25-41, (1978).2


\bibitem[Xi]{Xi} M. Xiao,  A theorem on Hermitian rank and mapping problems, Math. Res. Lett. https://mathweb.
ucsd.edu/~m3xiao/Hermitian-rank-04-02-2021.pdf

\bibitem[XY]{XY} Xiao, M., Yuan, Y.: Holomorphic maps from the complex unit ball to Type IV classical domains. {\it J. Math. Pures Appl. } {\bf 133} (2020), 139-166.

\bibitem[Zh]{Zh} Y. Zhang,  Rigidity and holomorphic Segre transversality for holomorphic segre maps. { Math. Ann. } {\bf 337} (2007), 457--478.


\end{thebibliography}
\end{document}